\documentclass[a4paper,12pt]{article}

\usepackage{amsthm}{\normalsize }
\usepackage{amsmath}
\usepackage{mathtools,leftindex,tensor}
\usepackage[version=3]{mhchem}
\usepackage{amssymb}
\usepackage{booktabs}
\usepackage{multirow}
\usepackage{bm}
\usepackage{array}
\usepackage{latexsym}
\usepackage{float}
\usepackage{diagbox} 
\usepackage{threeparttable}
\usepackage[textwidth=18cm,textheight=20cm]{geometry}

\usepackage[usenames]{color}
\usepackage[colorlinks=true]{hyperref}
\definecolor{mygray}{gray}{0.9}
\definecolor{deeppink}{RGB}{255,20,147}
\definecolor{mygreen}{rgb}{0.05, 0.576, 0.03}
\definecolor{myred}{rgb}{0.768, 0.09, 0.09}

\usepackage{subfigure}
\usepackage{caption}
\usepackage{indentfirst}
\usepackage{biblatex}
\newtheorem{theorem}{Theorem}[section]
\newtheorem{proposition}[theorem]{Proposition}
\newtheorem{lemma}[theorem]{Lemma}
\newtheorem{corollary}[theorem]{Corollary}
\newtheorem{remark}[theorem]{Remark}
\newtheorem{definition}[theorem]{Definition}
\newtheorem{example}[theorem]{Example}

\newcommand{\ud}{\mathrm{d}}
\newcommand{\eqdefr}{=\mathrel{\mathop:}}
\newcommand{\eqdefl}{\mathrel{\mathop:}=}
\newcommand{\Q}{\mathbb{Q}}
\newcommand{\FF}{\mathcal{F}}
\newcommand{\f}{\mathbb{F}}
\newcommand{\R}{\mathbb{R}}
\newcommand{\PP}{\mathbb{P}}
\newcommand{\Cov}{\mathbf{Cov}}
\newcommand{\Var}{\mathbf{Var}}
\newcommand{\Corr}{\mathbf{Corr}}

\newcommand{\e}{\varepsilon}
\DeclareMathOperator{\inter}{int}
\DeclareMathOperator{\conv}{conv}
\DeclareMathOperator{\supp}{supp}
\DeclareMathOperator{\relint}{ri}
\DeclareMathOperator{\interior}{int}

\newcommand{\F}{\mathscr{F}}

\usepackage{graphicx} 
\title{\bf Iterated Poisson Processes for Catastrophic Risk Modeling in Ruin Theory}

\author{Dongdong Hu \thanks{Yiwu Industrial \& Commercial College, Yiwu, China. Email: \underline{hudongdong@ywicc.edu.cn}} \and Svetlozar T. Rachev \thanks{Department of Mathematics and Statistics, Texas Tech University, Lubbock, TX, USA. Email: \underline{Zari.Rachev@ttu.edu}}\and Hasanjan Sayit \thanks{Department of Financial and Actuarial Mathematics, Xi'an Jiaotong Liverpool University, Suzhou, China. Email: \underline{Hasanjan.Sayit@xjtlu.edu.cn}} \and Hailiang Yang \thanks{Department of Financial and Actuarial Mathematics, Xi'an Jiaotong Liverpool University, Suzhou, China. Email: \underline{Hailiang.Yang@xjtlu.edu.cn}}\and Yildiray Yildirim \thanks{Zicklin School of Business, Baruch College, The City University of New York (CUNY), New York, NY, USA. Email: \underline{yildiray.yildirim@baruch.cuny.edu}}}

\date{}

\begin{document}

\maketitle

\textbf{Abstract:} This paper studies the properties of the Multiply Iterated Poisson Process (MIPP), a stochastic process constructed by repeatedly time-changing a Poisson process, and its applications in ruin theory. Like standard Poisson processes, MIPPs have exponentially distributed sojourn times (waiting times between jumps). We explicitly derive the probabilities of all possible jump sizes at the first jump and obtain the Laplace transform of the joint distribution of the first jump time and its corresponding jump size. In ruin theory, the classical Cram\'er-Lundberg model assumes that claims arrive independently according to a Poisson process. In contrast, our model employs an MIPP to allow for clustered arrivals, reflecting real-world scenarios, such as catastrophic events. Under this new framework, we derive the corresponding scale function in closed form, facilitating accurate calculations of the probability of ruin in the presence of clustered claims. These results improve the modeling of extreme risks and have practical implications for insurance solvency assessments, reinsurance pricing, and capital reserve estimation. 


\section{Introduction}

The Poisson process is one of the most widely used stochastic models in areas such as insurance, finance, and engineering because of its simplicity and its ability to model random events occurring over time. Its popularity arises from its mathematical tractability and the assumption that events occur independently at a constant rate. However, in many real-world scenarios, the classical Poisson process proves inadequate because it assumes that events occur one at a time and independently. In practice, events often occur in clusters, especially under extreme conditions such as natural disasters or financial crises, where multiple claims or defaults can occur almost simultaneously. Such clustering cannot be captured by traditional Poisson models, which, therefore, underestimate the risk in these situations.

To address these limitations, various extensions of the Poisson process have been introduced. The compound Poisson process, where jumps represent aggregated events, is commonly employed in actuarial science and queueing theory (Kl\"uppelberg et al., 2004, \cite{Kluppelberg_2004}). Another significant line of development involves time-changing the Poisson process by a subordinator, resulting in more flexible dynamics of the occurrence of events
(Lee \& Whitmore, 1993, \cite{lee_whitmore_1993}; Crescenzo et al., 2015, \cite{Crescenzo2015}; Orsingher \& Polito, 2012a, \cite{orsingher2012}). Such time-changed (or subordinated) processes have also been referred to as time-changing stochastic processes (Volkonski, 1958, \cite{volkonski1958}; It\^o \& McKean, 1965, \cite{1965ito}) and have found applications in finance and other areas (Clark, 1973, \cite{clark1973}; Geman, 2005, \cite{Geman2005};  McKean, 2002, \cite{McKean2002}). Moreover, processes with fractional or Bernstein intertimes, as investigated by Meerschaert et al., (2011), \cite{meerschaert2011} and Orsingher \& Toaldo (2015), \cite{Orsingher_Toaldo_2015}, and space-fractional Poisson processes (Orsingher \& Polito, 2012b, \cite{orsingher2012space}; Orsingher, 2013, \cite{orsingher2013fractional}), have been developed to incorporate heavy-tailed waiting times and memory effects. Furthermore, Poisson processes time-changed by compound Poisson-Gamma subordinators have been studied by Buchak \& Sakhno (2017), \cite{Buchak-Sakhno}, and compound Poisson processes with Poisson subordinators have been examined by Crescenzo et al. (2015), \cite{Crescenzo2015}. Despite these advances, many existing models still lack an explicit mechanism to model the clustering of the occurence of events
during catastrophic or systemic episodes.

This paper introduces the Multiple Iterated Poisson Process (MIPP), a novel stochastic model constructed by iterating Poisson processes over each other. The MIPP not only captures the randomness of the occurrences, but also explicitly models their clustering, thereby providing a more realistic framework for scenarios where events occur in bursts rather than in isolation. This approach is particularly relevant in contexts such as ruin theory, where clustered claims can drastically affect an insurer’s surplus, and in the modeling of systemic credit risk, where correlated defaults can occur in waves. Previous work has examined the case $n=2$, often referred to as an iterated Poisson process, studying, e.g., its crossing and hitting times under various boundary conditions (see Crescenzo \& Martinucci, 2009, \cite{crescenzo2009}; Orsingher \& Toaldo, 2015, \cite{Orsingher_Toaldo_2015}; Buchak \& Sakhno, 2017, \cite{Buchak-Sakhno} for related results). The MIPP generalizes this construction to any $n$, enabling even richer clustering behavior.

The practical implications of the MIPP are substantial in insurance and finance. Ruin theory, a cornerstone of actuarial science, dates back more than a century to Lundberg (1903), \cite{Lundberg1903},  and has been further developed by Cramér (1969), \cite{Cramr1969}. Classical models, such as the Cram´er–Lundberg framework, assume the claims occur independently over time (Gerber, 1979, \cite{1979Gerber}). However, they cannot adequately represent catastrophic risks, where claims occur in clusters following extreme events such as hurricanes or earthquakes. Underestimating this clustering leads to insufficient capital reserves and increased risk of insolvency. By integrating MIPP into ruin theory and deriving the corresponding scale function, this paper offers tools for a more accurate assessment of ruin probabilities, complementing standard references in risk theory (Rolski et al., 1999, \cite{rolski1999}; Asmussen \& Albrecher, 2010, \cite{asmussen2010}).

In financial markets, defaults often occur in bursts during periods of stress. Systemic risks and correlated defaults challenge standard models. The MIPP provides a way to incorporate the clustering of defaults into credit risk assessments, improving pricing, risk management, and regulatory compliance in markets where instruments like Collateralized Debt Obligations (CDOs) and Credit Default Swaps (CDS) are sensitive to simultaneous defaults. Beyond insurance and finance, the MIPP can also be applied to problems in operations research, such as queueing systems and network traffic modeling, where sudden bursts of arrivals can overwhelm systems designed for more stable conditions.

This paper advances the theory by extending the classical Poisson process through multiple independent Poisson processes, producing the MIPP. We derive key properties of the MIPP, including the distribution of sojourn times, probability mass functions of jump sizes, and the Laplace transform of the joint distribution of the first jump time and size. These theoretical results provide a foundation for applying the MIPP in various risk models. The focus is on applying the MIPP to ruin theory, where deriving the scale function is central for calculating ruin probabilities in the presence of clustered events.

The remainder of this paper is organized as follows. Section 2 defines the MIPP and explores its fundamental properties, including its sojourn times and jump distributions.  Section 3 focuses on the application of the MIPP to ruin theory, presenting the derivation of the scale function. An appendix calculates the high moments of the MIPP.


\section{Definition and  properties of the MIPP}

We start with $n$ independent Poisson processes $N_t^1, N_t^2, \cdots, N_t^n,$ with the same intensity parameter $\lambda$. We write $U_t^{(1)}=N_t^1, U_t^{(2)}=N_t^2, \cdots, U_t^{(n)}=N_t^n$. For any positive integers $1\le k\le m\le n$ we define $ V^{(k, k)}_t=U_t^{(k)},  V^{(k, k+1)}_t=U^{(k+1)}(V^{(k, k)}_t),  V^{(k, k+2)}_t=U^{(k+2)}(V^{(k, k+1)}_t), \cdots,  V^{(k, m)}_t=U^{(m)}(V^{(k, m-1)}_t)$. From this definition it is clear that $V_t^{(k, m)}$ is the multiple iteration of the Poisson processes $N_t^k, N_t^{k+1}, \cdots, N_t^m$. We call the processes $V_t^{(k, m)}$ multiple iterated Poisson processes and we use the acronym MIPP for these processes. For notational simplicity, we put
\[
V_t^{(m)}=:V_t^{(1, m)}, \forall\; \;  m\geq 2.
\]
Our goal in this section is to study some properties of $V_t^{(n)}$.
\begin{remark}\label{rem1} The above definition shows that we have
\begin{equation}\label{kVn}
V_t^{(n)}=N^{n}_{V_t^{(n-1)}}, \; \; V^{(n)}_t={V}^{(2, n)}_{N_t^1}.
\end{equation}
Therefore the process $V_t^{(n)}$ is a special case of the models considered in  \cite{Orsingher_Toaldo_2015}. Namely $V_t^{(n)}$ corresponds to $N^{f}(t)$ in their paper with $H^{f}(t)$ being the subordinator $V_t^{(n-1)}$. Observe here that $V_t^{(2, n)}$ has the same distribution as $V_t^{(n-1)}$.    
\end{remark}

By conditioning on the time change $V_t^{(n-1)}$ in (\ref{kVn}), it is clear that the probability mass function of $V_t^{(n)}$ satisfies 
\begin{equation}\label{71}
\begin{split}
P(V_t^{(n)}=k)=\frac{\lambda^k}{k!}\sum_{j=0}^{+\infty}j^ke^{-\lambda j}P(V_t^{(n-1)}=j), \; k\geq 0.
\end{split}
\end{equation}

From Theorem 2.1 of \cite{Orsingher_Toaldo_2015}, it follows that the difference-differential equations
governing the state probabilities of $V_t^{(n)}$ satisfy

\begin{equation}\label{08281}
\begin{split}
 \frac{d}{dt}P(V_t^{(n)}=k)&=-f_n(\lambda)P(V^{(n)}_t=k)+\sum_{m=1}^k\frac{\lambda^m}{m!}P(V_t^{(n)}=k-m)\int_0^{\infty}e^{-s\lambda}s^m\nu_{n-1}(ds)\\
 &=-\sum_{j=0}^{\infty}\lambda(1-e^{-\lambda j})P(V_1^{(n-2)}=j)P(V_t^{(n)}=k)\\
 &+\sum_{m=1}^k\sum_{j=0}^{\infty}\frac{\lambda^{m+1}e^{-\lambda j}}{m!}j^mP(V_1^{(n-2)}=j)P(V_t^{(n)}=k-m),
\end{split}
\end{equation}
where 
\begin{equation*}
\begin{split}
 f_n(u)&=\int_0^{\infty}(1-e^{-us})\nu_{n-1}(ds)=\sum_{j=0}^{\infty}\lambda(1-e^{-uj})P(V_1^{(n-2)}=j)\\
\end{split}
\end{equation*}
and
\begin{equation*}
\begin{split}
 \int_0^{\infty}e^{-s\lambda}s^m\nu_{n-1}(ds)=\sum_{j=0}^{\infty}\lambda e^{-\lambda j}j^mP(V_1^{(n-2)}=j).
\end{split}
\end{equation*}
Note here that  $\nu_{n-1}(j)=\lambda P(V_1^{(n-2)}=j)$ is the L\'evy measure of $V_t^{(n-1)}$ (see Proposition \ref{2.4} below for this). In the following subsection we study some properties of the MIPP.

\subsection{Martingales associated with an MIPP} 
Clearly the process $V_t^{(n)}$ is a subordinator, a non-negative L\'evy process with increasing sample paths. We put $\ell_1(\theta)=:\lambda(e^{\theta}-1)$ which is the characteristic exponent of the Poisson process. We also denote by $\ell_n(\theta), n\geq 1$,  the characteristic exponents of $V_t^{(n)}, n\geq 1$. The following relation can easily be verified
\begin{equation}\label{ell}
\ell_n(\theta)=\lambda e^{\ell_{n-1}(\theta)}-\lambda, \; \; n\geq 2.
\end{equation}
To see this observe that
\begin{equation}
\begin{split}
e^{t\ell_n(\theta)}=&Ee^{\theta V_t^{(n)}}=\sum_{k=0}^{+\infty}Ee^{\theta V_k^{(2,n)}}P(N^1_t=k)=\sum_{k=0}^{+\infty}(Ee^{\theta V_1^{(n-1)}})^kP(N^1_t=k)\\
=&\sum_{k=0}^{+\infty}(e^{k\ell_{n-1}(\theta)}P(N^1_t=k)=e^{t(\lambda e^{\ell_{n-1}(\theta)}-\lambda)},
\end{split}
\end{equation}
and then (\ref{ell}) follows from this.

\begin{proposition}\label{2.4} We have the following
\begin{enumerate}
\item [(a)] We have $EV_t^{(n)}=\lambda^nt$ and $M_t=:V_t^{(n)}-\lambda^nt$ is a martingale process.
\item [(b)] The process $e^{\theta V_t^{(n)}-t\ell_n(\theta)}$ is a martingale for any $\theta$. Define an equivalent probability measure by  $\frac{dP^{\theta}}{dP}=e^{\theta V_t^{(n)}-t\ell_n(\theta)}$ for each $\theta$.
Then under the measure $P^{\theta}$, the process $V_t^{(n)}$ has the characteristic exponent
\[
\ell_n^{\theta}(z)=:\ell_n(z+\theta)-\ell_n(\theta).
\]
\item [(c)] The process $M_t^2-\frac{1-\lambda^n}{1-\lambda}\lambda^nt$ is a martingale. 
\item [(d)] We have $V_t^{(n)}\overset{d}{=}\sum_{i=1}^{N_t^1}Y_i$, where $\{Y_i\}$ are i.i.d random variables that are independent of $N_t^1$ and have distribution $Y_i\overset{d}{=}V_1^{(n-1)}$.
The L\'evy measure of $V_t^{(n)}$, which we denote by $\nu_n(\cdot)$,  has support on the non-negative integers and  is given by $\nu_n(k)=\lambda P(V_1^{(n-1)}=k), k\geq 0$. Under the measure $P^{\theta}$, the L\'evy measure of $V^{(n)}_t$ is $\nu_n^{\theta}(k)=\lambda e^{k\theta}P(V_1^{(n-1)}=k), k\geq 0$.
\item[(e)] We have 
\[
\lim_{t\rightarrow +\infty}\frac{V_t^{(n)}}{t}=\lambda^n.
\]
\end{enumerate}
\end{proposition}

\begin{proof} The proofs are straightforward.
\end{proof}

\begin{remark}\label{remdel} The following processes 
\begin{equation}\label{66}
M_t^{(n)}=e^{\alpha t+\beta V_t^{(n)}}+(-\alpha+\lambda-\lambda Ee^{\beta V_1^{(n-1)}})\int_0^{t}e^{\alpha s+\beta V_s^{(n)}}ds    
\end{equation}
are also martingales for all $\alpha, \beta \in R$. This follows from Proposition 3.1 (see example 4.1 of the same paper) of  \cite{Delbaen-Haez1986}. If we pick $\alpha$ and $\beta$ in such a way that 
\[
\alpha+\lambda-\lambda Ee^{\beta V_1^{(n-1)}}=0,
\]
then the corresponding process $M_t^{(n)}=e^{-\alpha t+\beta V_t^{(n)}}$ is a martingale. Note here that  \cite{Delbaen-Haez1986} requires $\alpha$ to be positive. However, in our case we can easily check that $EM_t^{(n)}=1$ for all real numbers $\alpha, \beta$, which shows that $M_t^{(n)}$ is a martingale, since it is a L\'evy process. In fact we have
\begin{equation}
\begin{split}
EM_t^{(n)}=&e^{\alpha t}e^{t\ell_n(\beta)}+(-\alpha+\lambda-\lambda e^{\ell_{n-1}(\beta)})\int_0^te^{[\alpha +\ell_n(\beta)]s}ds\\
&=e^{(\alpha+\ell_n(\beta))t}-(\alpha+\ell_n(\beta))\frac{1}{\alpha+\ell_n(\beta)}[e^{(\alpha+\ell_n(\beta))t}-1]\\
&=1.
\end{split}
\end{equation}
Note that in the above equality,  we have used the property (\ref{ell}).
\end{remark}

\begin{remark} (The governing equation) The governing equation can also be written as follows also.
\begin{equation}
\begin{split}
\frac{d}{dt}P(V_t^{(n)}=k)=&-\lambda P(V_t^{(n)}=k)+\lambda P(V_{N^1_t+1}^{(2, n)}=k)\\
=& -\lambda P(V_t^{(n)}=k)+\lambda \sum_{j=0}^kP(V_t^{(n)}=j)P(V_1^{(n-1)}=k-j)\\
=&-\lambda P(V_t^{(n)}=k)+\lambda \sum_{j=0}^kP(V_t^{(n)}=k-j)P(V_1^{(n-1)}=j)\\
=&-\lambda P(V_t^{(n)}=k)+\int_0^kP(V_t^{(n)}=k-s)\nu_n(ds),
\end{split}
\end{equation}
where $\nu_n(\cdot)$ is the L\'evy measure of $V^{(n)}_t$.
\end{remark}

\subsection{ Jump times of an MIPP}
We define the jump times $J_0^{(n)}, J_1^{(n)}, \cdots, $ of $V_t^{(n)}$ by
\[
J_0^{(n)}=0, \; J_{k+1}^{(n)}=\inf\{t\geq J_k^{(n)}: V_t^{(n)}\neq V_{J_k^{(n)}}^{(n)}\},  \; \; k\geq 0,
\]
with the convension $inf\; \emptyset=\infty$.  The holding times, also called sojourn times,  are defined for all $k\geq 1$ by
\begin{equation}\label{sojourn}
 S_k^{(n)}=\left \{
 \begin{array}{cc}
  J_k^{(n)}-J_{k-1}^{(n)} & \mbox{if}\; \; J^{(n)}_{k-1}<\infty, \\
  \infty   & \mbox{otherwise}.
 \end{array}
 \right.
\end{equation}
The right-continuity of $V_t^{(n)}$ implies that $S^{(n)}_k>0$ almost surely for all $k\geq 1$. Our main result in this section is as follows. In this theorem, we use the nature of  $V^{(n)}_t$ being a compound Poisson process  (see part (d) of Proposition \ref{2.4} above). The main idea that we have used in the proof of this theorem is that the first jump time $J^{(n)}_1$ of $V_t^{(n)}$ is equal to  the first jump time of the Poisson process $N_t^1$ on the event $\{Y_1>0\}$, but equals the second jump time of $N_t^1$ on the event $\{Y_1=0\}\cap \{Y_2>0\}$, and so on.

\begin{proposition}\label{sojorn} For each fixed $n\geq 2$, the sojourn times defined in (\ref{sojourn}) are i.i.d exponential random variables. More specifically,
\[
S_k^{(n)}\sim Exp(\lambda q_{n-1}), \forall k\geq 1, 
\]
where $q_{n-1}=P(V_1^{(n-1)}>0)$. The sequence $q_j, j\geq 1,$ satisfy  the  recursive relation
\[
q_j=1-e^{-\lambda q_{j-1}}, \; j\geq 2,
\]
with
\[
q_1=1-e^{-\lambda}.
\]
\end{proposition}

\begin{proof} From part (d) of Proposition \ref{2.4}  we have $V_t^{(n)}\overset{d}{=}\sum_{k=1}^{N_t^1}Y_k$, where the $Y_k$ are i.i.d. random variables with  $Y_i\overset{d}{=}V_1^{(n-1)}$. Let $\tau^{(1)}_k$ denote the jump times of $N_t^1$. We clearly have $P(J^{(n)}_1=\tau_k^{(1)},k\geq 1)=1$ as the first jump time of $V_t^{(n)}$  happens only in one of the jump times of $N_t^1$. Also observe that $J_1^{(n)}=\tau_k^{(1)}$ on the event $\{Y_1=0\}\cap \{Y_2=0\}\cdots \{Y_{k-1}=0\}\cap \{Y_k>0\}$ for each $k\geq 1$. Therefore we have
\begin{equation*}
\begin{split}
 J_1^{(n)}=\tau_1^{(1)}1_{\{Y_1>0\}}+\tau_2^{(1)}1_{\{Y_1=0\}\cap\{Y_2>0\}}+\tau_3^{(1)}1_{\{Y_1=0\}\cap\{Y_2=0\}\cap\{Y_3>0\}}+\cdots.
\end{split}
\end{equation*}
Then, since $q_{n-1}=P(V_1^{(n-1)}>0)$ the moment generating function of $J^{(n)}_1$ is
\begin{equation*}
\begin{split}
 E[e^{\theta J_1^{(n)}}]&=E[e^{\tau_1^{(1)}}]P(Y_1>0)+E[e^{\tau_2^{(1)}}]P(Y_1=0)P(Y_2>0)\\
 &+E[e^{\tau_3^{(1)}}]P(Y_1=0)P(Y_2=0)P(Y_3>0)+\cdots\\
 &=\frac{\lambda}{\lambda-\theta}q_{n-1}+(\frac{\lambda}{\lambda-\theta})^2(1-q_{n-1})q_{n-1}+(\frac{\lambda}{\lambda-\theta})^3(1-q_{n-1})^2q_{n-1}+\cdots\\
 &=\frac{\lambda}{\lambda-\theta}q_{n-1}\sum_{n=0}^{\infty}(\frac{\lambda(1-q_{n-1})}{\lambda-\theta})^n\\
 &=\frac{\lambda q_{n-1}}{\lambda q_{n-1}-\theta}.
\end{split}
\end{equation*}
The last equation assumes the parameter $\theta$ satisfies  $|\frac{\lambda(1-q_{n-1})}{\lambda-\theta}|<1$. This result shows that $\tau^{(n)}_1$ is a exponential random variable with parameter $\lambda q_{n-1}$, i.e., the sojourn times satisfy
\[
S_k^{(n)}\sim Exp(\lambda q_{n-1}), \forall \;\;\; k\geq 1.
\]

Next, note that
\begin{equation*}
\begin{split}
 \{V_t^{(n)}=0\}=\{\tau_1^{(1)}>t\}\cup\{\tau_1^{(1)}\leq t<\tau_2^{(1)}\text{ and }Y_1=0\}\cup\{\tau_2^{(1)}\leq t<\tau_3^{(1)}\text{ and }Y_1=0,Y_2=0\}\cup\cdots.
\end{split}
\end{equation*}
Therefore 
\begin{equation*}
\begin{split}
 &P(V_t^{(n)}=0)\\
 &=P(\tau_1^{(1)}>t)+P(\tau_1^{(1)}\leq t<\tau_2^{(1)})P(Y_1=0)+P(\tau_2^{(1)}\leq t<\tau_3^{(1)})P(Y_1=0)P(Y_2=0)+\cdots\\
 &=P(\tau_1^{(1)}>t)+P(\tau_1^{(1)}\leq t,S_1^{(1)}>t-\tau_1^{(1)})P(Y_1=0)\\
 &+P(\tau_2^{(1)}\leq t,S_2^{(1)}>t-\tau_2^{(1)})P(Y_1=0)P(Y_2=0)+\cdots\\
 &=\int_t^{\infty}\lambda e^{-\lambda x}dx+(1-q_{n-1})\int_0^t\int_{t-x}^{\infty}\lambda e^{-\lambda x}\lambda e^{-\lambda y}dydx+(1-q_{n-1})^2\int_0^t\int_{t-x}^{\infty}\lambda^2 xe^{-\lambda x}\lambda e^{-\lambda y}dydx+\cdots,
\end{split}
\end{equation*}
where the third equation is because $J_k^{(n)}$ follows the Erlang distribution with parameters $\lambda$ and $k$ and $S_k^{(1)}$ follows the exponential distribution with parameter $\lambda$. Thus 
\begin{equation*}
\begin{split}
 1-q_n=e^{-\lambda t}+(1-q_{n-1})\lambda te^{-\lambda t}+\frac{1}{2}(1-q_{n-1})^2\lambda^2t^2e^{-\lambda t}+\frac{1}{6}(1-q_{n-1})^3\lambda^3t^3e^{-\lambda t}+\cdots.
\end{split}
\end{equation*}
By letting $t=1$, we get the iteration formula
\begin{equation*}
\begin{split}
 1-q_n&=e^{-\lambda}+(1-q_{n-1})\lambda e^{-\delta}+\frac{1}{2}(1-q_{n-1})^2\lambda^2e^{-\delta}+\frac{1}{6}(1-q_{n-1})^3\lambda^3e^{-\delta}+\cdots\\
 &=e^{-\lambda}\sum_{n=0}^{\infty}\frac{((1-q_{n-1})\lambda)^n}{n!}=e^{-\lambda q_{n-1}}.
\end{split}
\end{equation*}
Lastly,  it is clear that $q_1=1-e^{-\lambda}$ as this case corresponds to a standard Poisson process.
\end{proof}

Recall that in the case of standard Poisson process $N_t$ we have $N_{\tau_1}=1$, where $\tau_1$ is the first jump time of $N_t$. We would like to determine the distribution of $V^{(n)}_{J_1}$. We have the following result.

\begin{proposition}\label{jjj} For each integer $k\geq 1$ and for any $n\geq 2$ 
\begin{equation}\label{jtimes}
P(V^{(n)}_{J_1^{(n)}}=k)=\frac{P(V_1^{(n-1)}=k)}{P(V_1^{(n-1)}\geq 1)},
\end{equation}
where $J_1^{(n)}$ is the first jump time of $V_t^{(n)}$.
\end{proposition}
\begin{proof} For any given twice continuously differentiable function $f$ (in fact continuous functions are enough as $V_t^{(n)}$ has paths of bounded variation) the following process
\begin{equation}\label{Mf}
M_t^f=:f(V_t^{(n)})-f(0)-\int_0^t\mathcal{A}f(V_s^{(n)})ds
\end{equation}
is a martingale, where $\mathcal{A}$ is the infinitesimal generator of $V_t^{(n)}$ and it is given by
\begin{equation}
\begin{split}
\mathcal{A}f(x)&= \int[f(x+y)-f(x)] d\nu_n(y)\\
&=\sum_{i=0}^{+\infty}[f(x+i)-f(x)]\nu_n(i).
\end{split}
\end{equation}
Here $\nu_n$ is the L\'evy measure of $V_t^{(n)}$. We have $EM^{f}_{J_1^{(n)}}=0$  as $M_t^{f}$ is a martingale. Then from (\ref{Mf}) we obtain
\begin{equation}\label{fj}
Ef(V^{(n)}_{J_1^{(n)}})-f(0)=\sum_{i=0}^{+\infty}E\int_0^{J_1^{(n)}}[f(V_{s-}^{(n)}+i)-f(V_{s-}^{(n)})]ds\nu_n(i).
\end{equation}
Since $J_1^{(n)}$ is the first jump time of $V_t^{(n)}$ we have $V_{s-}^{(0)}=0$ on $[0, J_1^{(n)})$. Note that when $i=0$ we have $[f(V_{s-}^{(n)}+i)-f(V_{s-}^{(n)})]=0$. Therefore (\ref{fj}) reduces to
\begin{equation}\label{13i}
\begin{split}
Ef(V^{(n)}_{J_1^{(n)}})-f(0)=&[EJ_1^{(n)}]\sum_{i=1}^{+\infty}[f(i)-f(0)]\nu^{(n)}(i)\\
=&[EJ_1^{(n)}]\sum_{i=1}^{+\infty}f(i)\nu^{(n)}(i)- f(0)[EJ_1^{(n)}]\sum_{i=1}^{+\infty}\nu^{(n)}(i).
\end{split}
\end{equation}
Now, in Proposition \ref{sojorn} we have shown that $J_1^{(n)}\sim Exp(\lambda q_{n-1})$. Therefore $EJ_1^{(n)}=\frac{1}{\lambda q_{n-1}}=\frac{1}{\lambda P(V_1^{(n-1)}\geq 1)}$. Also since $V_{J_1^{(n)}}^{(n)}$ assumes only positive integers we have $Ef(V^{(n)}_{J_1^{(n)}})=\sum_{i=1}^{+\infty}f(i)P(V^{(n)}_{J_1^{(n)}}=i)$. Hence (\ref{13i}) becomes
\begin{equation}\label{14i}
\sum_{i=1}^{+\infty}f(i)P(V^{(n)}_{J_1^{(n)}}=i)=\frac{1}{\lambda P(V_1^{(n-1)}\geq 1)}\left [\sum_{i=1}^{+\infty}f(i)\nu^{(n)}(i)- f(0)\sum_{i=1}^{+\infty}\nu^{(n)}(i)\right ].   
\end{equation}
The relation (\ref{14i}) is true for any twice continuously differentiable function on $(-\infty, +\infty)$. Then for each fixed $i\geq 1$ we can pick $f$ in such a way that $f(i)\neq 0$ and $f(k)=0$ for all other positive integers $k$. With such $f$ and (\ref{14i}) we obtain
\[
f(i)P(V^{(n)}_{J_1^{(n)}}=i)=\frac{f(i)\nu_n(i)}{\lambda P(V_1^{(n-1)}>0)}
\]
which gives 
\[
P(V^{(n)}_{J_1^{(n)}}=i)=\frac{\nu_n(i)}{\lambda P(V_1^{(n-1)}\geq 1)}.
\]
From part (d) of Proposition (\ref{2.4}) we have $\nu_n(i)=\lambda P(V^{(n-1)}_1=i)$ and this ends the proof.
\end{proof}

\begin{corollary} When $n=2$ we have $q_1=P(N_1>0)=1-e^{-\lambda}$. Hence $S_k^{(2)}\sim Exp[\lambda (1-e^{-\lambda})]$. We also have 
\[
P(V^{(2)}_{J^{(2)}_1}=k)=\frac{\lambda^k}{k!}\frac{e^{-\lambda}}{1-e^{-\lambda}},
\]
for all integers $k\geq 1$.    
\end{corollary}
\begin{proof} This follows easily from Proposition \ref{jjj}.
    
\end{proof}

\begin{proposition} The Joint Laplace transform of $Q_n=:(J_1^{(n)},V_{J_1^{(n)}}^{(n)} )$
is given by
\[
\mathcal{L}_{Q_n}(s)=1-(s_1+\ell_n(s_2))\frac{1}{(\lambda q_{n-1}+s_1)},
\]
for all $s=(s_1, s_2)$ with $s_1\neq 0$. Here $q_{n-1}$ is given in Proposition \ref{sojorn} above.
\end{proposition}
\begin{proof} Note that $M_t^{(n)}$ in Remark \ref{remdel} is a martingale with $EM_t^{(n)}=1$. Therefore  $EM_{J_1^{(n)}}=1$.
This implies that  
\[
1=Ee^{\alpha J_1^{(n)}+\beta V_{J_1}^{(n)}}-(\alpha +\ell_n(\beta))E\int_0^{J_1^{(n)}}e^{\alpha s+\beta V_s^{(n)}}ds.
\]
Since $V_s^{(n)}=0$ in the interval $[0, J_1^{(n)})$ we have 
\begin{equation}\label{16i}
\begin{split}
1=&Ee^{\alpha J_1^{(n)}+\beta V_{J_1}^{(n)}}-(\alpha +\ell_n(\beta))E\int_0^{J_1^{(n)}}e^{
\alpha s}ds\\
=&Ee^{\alpha J_1^{(n)}+\beta V_{J_1}^{(n)}}-(\alpha+\ell_n(\beta))\frac{1}{\alpha}(Ee^{\alpha J_1^{(n)}}-1).
\end{split}
\end{equation}
Note here that $Ee^{\alpha J_1^{(n)}}$ is the Laplace transform of the exponential random variable $J_1^{(n)}\sim exp(\lambda q_{n-1})$. Therefore $Ee^{\alpha J_1^{(n)}}=\frac{\lambda q_{n-1}}{\lambda q_{n-1}+\alpha}$. Then from (\ref{16i}) we obtain
\[
Ee^{\alpha J_1^{(n)}+\beta V_{J_1}^{(n)}}=1-\frac{\alpha+\ell_n(\beta)}{\lambda q_{n-1}+\alpha}.
\]
We replace $\alpha, \beta $ by the Laplace parameters $s_1=:\alpha, s_2=\beta$ and obtain the expression for $\mathcal{L}_{Q_n}(s)$. 
\end{proof}

\section{Applications in Ruin theory}
A L\'evy process $X_t$ is said to be spectrally negative if the corresponding L\'evy measure $\Pi$ has support on $(-\infty,0)$ which means $\Pi$ assigns measure zero to $[0, +\infty)$. This means that the jumps of the L\'evy process $X_t$ are all negative. For a spectrally negative L\'evy process $X_t$, 
the random variables $e^{\lambda X_t}$ are integrable at least for all $\lambda\geq 0$ (this is due to the fact that it does not have positive jumps) and its Laplace exponent $\psi(\lambda)$, which is defined by the relation
\begin{equation*}
\begin{split}
 E[e^{\lambda X_t}]=e^{\psi(\lambda)t},
\end{split}
\end{equation*}
takes the form
\begin{equation*}
\begin{split}
 \psi(\lambda)=-\mu\lambda+\frac{\sigma^2}{2}\lambda^2+\int_{(-\infty,0)}(e^{\lambda x}-1-\lambda x\bm{1}_{\{|x|<1\}})\Pi(dx),
\end{split}
\end{equation*}
where $\mu , \sigma \in R,$ and $\Pi$ is a measure on $(-\infty, 0)$ with $\int_{-1}^0 x^2\Pi(dx)<\infty$. If $\Pi$ satisfies $\int_{-1}^0|x|\Pi(dx)|<\infty$ then the jump part of $X_t$ is a process with  finite variation. It is well known that (see \cite{cohen2013theory}), $\psi(\lambda)$ is strictly convex on $[0, +\infty)$ and $\lim_{\lambda \rightarrow +\infty}\psi(\lambda)=+\infty$. Also if 
$\psi'(0^+)>0$ then $\lim_{t\rightarrow +\infty}X_t=+\infty$ while if  
$\psi'(0^+)<0$, then $\lim_{t\rightarrow +\infty}X_t=-\infty$. In comparison,  $\psi'(0^+)=0$ implies $\limsup_{t\rightarrow +\infty}X_t=-\liminf_{t\rightarrow +\infty}X_t=+\infty$, which means that the paths of $X_t$ oscillate. For any  $q\geq0$ define
\begin{equation}\label{phi}
\begin{split}
 \Phi(q)=\sup\{\lambda\geq0:\psi(\lambda)=q\},
\end{split}
\end{equation}
to be the largest root of the equation $\psi(\lambda)=q$. We write down the following definition of the scale function, which is definition 1.1 in  \cite{cohen2013theory}.
\begin{definition}\label{def4.1} For any 
spectrally negative L\'evy process $X$ with Laplace exponent $\psi$ and for any real number $q\geq 0$ the scale function $W^{(q)}(x)$ (the q-scale function) is defined to be $W^{(q)}(x)=0$ on $(-\infty, 0)$ and on $[0, +\infty)$ it is defined to be the unique right-continuous function with Laplace transform 
\begin{equation*}
\begin{split}
 \int_0^{\infty}e^{-\beta x}W^{(q)}(x)dx=\frac{1}{\psi(\beta)-q}
\end{split}
\end{equation*}
for all $\beta>\Phi(q)$. 
\end{definition}
As in \cite{cohen2013theory}, we denote the scale function $W^{(0)}(x)$ by $W(x)$ and call it the scale function (instead of the 0-scale function). The following theorem, which is Theorem 1.2 in \cite{cohen2013theory}, explains the importance of the scale functions.
\begin{theorem} (Theorem 1.2 in \cite{cohen2013theory})
Define
\begin{equation*}
\begin{split}
 \tau_a^+&=\inf\{t>0:X_t>a\},\\
 \tau_0^-&=\inf\{t>0:X_t<0\}.
\end{split}
\end{equation*}
all $q\geq0, a>0,$ and $x<a$. Then 
\begin{equation*}
\begin{split}
 E_x[e^{-q\tau_a^+}\bm{1}_{\{\tau_a^+<\tau_0^-\}}]=\frac{W^{(q)}(x)}{W^{(q)}(a)}.
\end{split}
\end{equation*}
\end{theorem}

In this section we consider the following spectraly negative L\'evy process and compute the corresponding scale function analytically: 

\begin{equation}\label{RR}
\begin{split}
 R_t=ct-\sum_{i=0}^{V_t^{(2)}}\xi_i+\sigma W_t,
\end{split}
\end{equation}
In this model, the  $\xi_i$ are i.i.d  exponentially distributed random variables with density
$g(x)=\delta e^{-\delta x}\cdot1_{[0,+\infty)}(x)$. We use the notation $P_x$ to denote the law of the process $x+R_t$ for any constant $x$. The corresponding expectation value is denoted by $E_x$. When $x=0$ we use $P$ instead of $P_0$ for the law of $R_t$ and the corresponding expected value is denoted by $E$.

In this section we calculate the scale function of $R_t$ analytically.
\begin{remark} We have 
\begin{equation*}
\begin{split}
 E_xR_1&=x+c-E[\sum_{i=0}^{V^{(2)}_1}\xi_i]=c-\sum_{k=0}^{\infty}E[\sum_{i=0}^{N_k^2}\xi_i]\frac{\lambda^k}{k!}e^{-\lambda}=x+c-\sum_{k=0}^{\infty}\sum_{j=0}^{\infty}(E[\xi])^j\frac{(\lambda k)^j}{j!}e^{-\lambda k}\frac{\lambda^k}{k!}e^{-\lambda}\\
 &=x+c-e^{-\lambda}\sum_{k=0}^{\infty}\frac{\lambda^k}{k!}e^{-\lambda k}\sum_{j=0}^{\infty}\frac{(\lambda k)^j}{j!}(E[\xi])^j=c-e^{-\lambda}\sum_{k=0}^{\infty}\frac{\lambda^k}{k!}e^{-\lambda k}e^{\lambda kE[\xi]}\\
 &=x+c-e^{-\lambda+\lambda e^{-\lambda+\frac{\lambda}{\delta}}},
\end{split}
\end{equation*}
\end{remark}

\begin{remark} 
Let $Y_k$ be i.i.d random variables with 
\[
Y_k\sim \sum_{i=0}^{N_1^{(2)}}\xi_i, \; \forall k\geq 1.
\]
Then 
\[
R_t\overset{d}{=}x+ct-\sum_{k=0}^{N_t^1}Y_k+\sigma W_t,
\]
where $N_t^1$ is standard Poisson process with intensity $\lambda$. Hence the standard results on risk theory apply to the case of $R_t$. In this section we obtain the probability of ruin by applying direct approach instead of applying the traditional results to the case of $R_t$. As the distribution of $Y_k$ is complex, traditional results would not lead to an analytical expression for the scale function for $R$. But in our direct approach we are able to obtain the scale function analytically.
\end{remark}

The characteristic exponent of this process $R_t$ is given by 
\begin{equation*}
\begin{split}
\psi_R(\theta)=c\theta-\lambda+\lambda e^{-\lambda+\lambda\frac{\delta}{\delta+\theta}}+\frac{1}{2}\sigma^2\theta^2.
\end{split}
\end{equation*}
The first order derivative is 
\[
\psi_R'(\theta)=c+\sigma^2\theta-\lambda^2 e^{-\lambda+\lambda\frac{\delta}{\delta+\theta}}\frac{\delta}{(\delta+\theta)^2},
\]
and then
\[
\lim_{\theta \rightarrow 0}\psi_R'(\theta)=c-\frac{\lambda^2}{\delta}.
\]
Therefore the condition $c\delta>\lambda^2$ guarantees that
$\lim_{\theta \rightarrow 0}\psi_R'(0)>0$.
We impose the following assumption on the model (\ref{RR}).

\textbf{Assumption 1:} The model (\ref{RR}) satisfies $c\delta>\lambda^2$. 

As stated in the paragraph before Definition \ref{def4.1}, Assumption 1  guaranties that $\lim_{t\rightarrow +\infty}R_t=+\infty$. We put $\overline{R}_t=\sup_{s\in [0, t]}R_s$ and $\underline{R}_t=\inf_{s\in [0, t]}R_s$. Also we denote by
\begin{equation}
 \psi_x=P_x(\underline{R}_{\infty}\geq 0),   
\end{equation}
the probability that the risk process $R_t$ never goes below the zero during its lifetime.
From the proof of Theorem 2.1 of \cite{cohen2013theory} we have the following relation.
\begin{equation}
 \psi_x=W_R(x)\psi_R'(0)=(c-\frac{\lambda^2}{\delta})W_R(x),   
\end{equation}
where $W_R(x)$ is the scale function of $R$. In the following result we calculate the scale functions of $R$ explicitly.

\begin{proposition}\label{prop4.5} Consider the model $R_t$ with $c>0$. Define
\[
\digamma(z)=\frac{\lambda(1-e^{-\lambda})+q}{c}\pi(z)-\frac{\lambda e^{-\lambda}}{c}(\pi*G)(z),
\]
 where $G(x)=\frac{e^{-\delta x}\sqrt{\delta\lambda}\cdot I_1(2\sqrt{\delta\lambda x})}{\sqrt{x}}$, $I_1(2\sqrt{-\delta\lambda x})$ is the modified Bessel function of the first kind with parameters $1$ and $2\sqrt{-\delta\lambda x}$, and $\pi(x)=1$ on $[0, +\infty)$ and zero otherwise. The scale function of $R_t$ is then
\begin{equation*}
\begin{split}
 W_R^{(q)}(x)=\pi *\sum_{n=0}^{\infty}\digamma^{* n}*F_{n+1}(x),
\end{split}
\end{equation*}
where $F_{n+1}(x)=\frac{(\frac{2c}{\sigma^2})^{n+1}x^n}{n!}e^{-\frac{2c}{\sigma^2}x}$ is the probability density function of $Erlang(n+1,2c/\sigma^2)$.

\end{proposition}

\begin{proof}
\begin{equation*}
\begin{split}
 E[e^{\theta(R_t-x)}]=e^{t\psi_R(\theta)},
\end{split}
\end{equation*}
where
\begin{equation*}
\begin{split}
 \psi_R(\theta)=c\theta-\lambda+\lambda e^{-\lambda+\lambda\frac{\delta}{\delta+\theta}}+\frac{1}{2}\sigma^2\theta^2.
\end{split}
\end{equation*}
Then, 
\begin{equation*}
\begin{split}
 \int_0^{\infty}e^{-\theta x}W^{(q)}(x)dx&=\frac{1}{\psi_R(\theta)-q}=\frac{1}{\theta}\frac{1}{c-\frac{\lambda(1-e^{-\lambda+\lambda\frac{\delta}{\delta+\theta}})+q}{\theta}+\frac{1}{2}\sigma^2\theta}\\
 &=\frac{1}{c\theta}\frac{c/(c+\frac{1}{2}\sigma^2\theta)}{1-\frac{\lambda(1-e^{-\lambda+\lambda\frac{\delta}{\delta+\theta}})+q}{c\theta}\frac{c}{c+\frac{1}{2}\sigma^2\theta}}\\
 &=\frac{1}{\theta}\sum_{n=0}^{\infty}(\frac{\lambda(1-e^{-\lambda})+q}{c}\frac{1}{\theta}-\frac{\lambda e^{-\lambda}}{c}\frac{1}{\theta}(e^{\frac{\lambda\delta}{\delta+\theta}}-1))^n(\frac{c}{c+\frac{1}{2}\sigma^2\theta})^{n+1},
\end{split}
\end{equation*}
as long as
\begin{equation*}
\begin{split}
 \left|\frac{\lambda(1-e^{-\lambda+\lambda\frac{\delta}{\delta+\theta}})+q}{c\theta}\frac{c}{c+\frac{1}{2}\sigma^2\theta}\right|<1.
\end{split}
\end{equation*}

The convolution theorem of the theory of the Laplace transformation  \cite{Polyanin} states that the product of the Laplace transforms of two functions equals  the Laplace transform of the convolution of these two functions. Applying this theorem and making use of  $\mathcal{L}^{-1}(\frac{1}{\theta})=1$, $\mathcal{L}^{-1}((\frac{c}{c+\frac{1}{2}\sigma^2\theta})^{n+1})=\mathcal{L}^{-1}((\frac{2c/\sigma^2}{2c/\sigma^2+\theta})^{n+1})=\frac{(\frac{2c}{\sigma^2})^{n+1}x^n}{n!}e^{-\frac{2c}{\sigma^2}x}$, which is the density function of the Erlang distribution with parameters  $(n+1,2c/\sigma^2)$, and 
\begin{equation}\label{keylap}
\begin{split}
 \mathcal{L}^{-1}(e^{\frac{\lambda\delta}{\delta+\theta}}-1)=\frac{e^{-\delta x}\sqrt{\delta\lambda}\cdot I_1(2\sqrt{\delta\lambda x})}{\sqrt{x}},
\end{split}
\end{equation}
we have
\begin{equation}\label{09152}
\begin{split}
 W^{(q)}(x)&=\pi*\sum_{n=0}^{\infty}(\frac{\lambda(1-e^{-\lambda})+q}{c}\pi-\frac{\lambda e^{-\lambda}}{c}\pi*G)^{*n}*F_{n+1}(x)\\
 &=\pi *\sum_{n=0}^{\infty}\digamma^{* n}*F_{n+1}(x).
\end{split}
\end{equation}
\end{proof}

\begin{remark} Consider the model 
\begin{equation*}
\begin{split}
 \hat{R}_t=x+ct-\sum_{i=0}^{V_t^{(2)}}\xi_i.
\end{split}
\end{equation*}
From the calculations in the proof of Proposition \ref{prop4.5} above, it is easy to see that the scale function for this model is 
\[\hat{W}_{\hat{R}}^{(q)}(x)=\pi*\sum_{n=0}^{\infty}\digamma^{* n}(x).
\]
This scale function can be obtained as the limit
of the scale function $W_R^{(q)}(x)$ when $\sigma \rightarrow 0$, i.e., $\hat{W}_{\hat{R}}^{(q)}(x)=\lim_{\sigma\rightarrow 0}W_R^{(q)}(x)$. To see this,  observe that
\begin{equation*}
 \lim_{\sigma\to0}F_{n+1}(x)=\left\{
 \begin{aligned}
 +\infty, && {x=0}\\
 0, && \text{otherwise},
 \end{aligned}
 \right.
\end{equation*}
and hence $F_{n+1}(x)$ converges to the Dirac delta function: $D(x)=+\infty$ when $x=0$ and $D(x)=0$ when $x\neq 0$. Since for any continuous function $f$ one has $f*D=f$, by taking the limit of the expression (\ref{09152}) one obtains the claim.
\end{remark}

The key observation that made the analytical calculation of the scale function for $R_t$ possible was the relation (\ref{keylap}). Below we use this relation to derive the scale function for a more complex risk process.To this end, suppose our risk process is given by 
\begin{equation*}
\begin{split}
 \tilde{R}_t=u+ct-\sum_{i=0}^{V_t^{(2)}}\eta_i+\sigma W_t,
\end{split}
\end{equation*}
where $\eta_i$ are i.i.d with density functions
$\eta_i(x)\sim \sum_{j=1}^N\alpha_j\delta_je^{-\delta_jx}$ (here $\sum_{j=1}^N\alpha_j=1$) for all $i\geq 1$ (a mixture of Exponential random variables). For the risk process $\tilde{R}_t$ we can calculate the scale function analytically also. First we need its Laplace exponent.
\begin{equation*}
\begin{split}
 E[e^{\theta(\tilde{R}_t-u)}]=e^{t\tilde{\psi}_R(\theta)},
\end{split}
\end{equation*}
where
\begin{equation*}
\begin{split}
 \tilde{\psi}_R(\theta)=c\theta-\lambda+\lambda e^{-\lambda+\sum_{j=1}^N\lambda\frac{\alpha_j\delta_j}{\delta_j+\theta}}+\frac{1}{2}\sigma^2\theta^2.
\end{split}
\end{equation*}
\begin{proposition} The scale function $\tilde{W}_{\tilde{R}}^{(q)}(x)$ of the risk process $\tilde{R}$ is given by
\begin{equation}\label{19}
 \tilde{W}^{(q)}_{\tilde{R}}(x)=\pi*\sum_{n=0}^{\infty}(\frac{\lambda(1-e^{-\lambda})+q}{c}\pi-\frac{\lambda e^{-\lambda}}{c}\pi*G)^{*n}*F_{n+1}(x),
\end{equation}
where
\begin{equation*}
\begin{split}
 G(x)=\sum_{i=1}^NG_i(x)+\sum_{i\neq j}^NG_i(x)*G_j(x)+\cdots+G_1(x)*\cdots*G_N(x).
\end{split}
\end{equation*}
with $G_i(x)\sim\frac{e^{-\delta x}\sqrt{\lambda\alpha_i\delta_i}\cdot I_1(2\sqrt{\lambda\alpha_i\delta_i x})}{\sqrt{x}}$. In (\ref{19}), $F_{n+1}(x)$ and $\pi(x)$ are defined as in Proposition \ref{prop4.5}.
\end{proposition}
\begin{proof}

Then, we have
\begin{equation*}
\begin{split}
 &\int_0^{\infty}e^{-\theta x}W^{(q)}(x)dx=\frac{1}{\psi_R(\theta)-q}\\
 &=\frac{1}{\theta}\frac{1}{c-\frac{\lambda(1-e^{-\lambda+\sum_{j=1}^N\lambda\frac{\alpha_j\delta_j}{\delta_j+\theta}})+q}{\theta}+\frac{1}{2}\sigma^2\theta}=\frac{1}{\theta}\frac{c/(c+\frac{1}{2}\sigma^2\theta)}{1-\frac{\lambda(1-e^{-\lambda+\sum_{j=1}^N\lambda\frac{\alpha_j\delta_j}{\delta_j+\theta}})+q}{c\theta}\frac{c}{c+\frac{1}{2}\sigma^2\theta}}\\
 &=\frac{1}{\theta}\sum_{n=0}^{\infty}\left(\frac{\lambda(1-e^{-\lambda})+q}{c}\frac{1}{\theta}-\frac{\lambda e^{-\lambda}}{c}\frac{1}{\theta}\right.\\
 &\left.\times(\sum_{i=1}^N(e^{\frac{\lambda\alpha_i\delta_i}{\delta_i+\theta}}-1)+\sum_{i\neq j}^N(e^{\frac{\lambda\alpha_i\delta_i}{\delta_i+\theta}}-1)(e^{\frac{\lambda\alpha_j\delta_j}{\delta_j+\theta}}-1)+\cdots+\prod_{i=1}^N(e^{\frac{\lambda\alpha_i\delta_i}{\delta_i+\theta}}-1))\right)^n(\frac{c}{c+\frac{1}{2}\sigma^2\theta})^{n+1}.
\end{split}
\end{equation*}
Since $\mathcal{L}^{-1}(\frac{1}{\theta})=1$, $\mathcal{L}^{-1}((\frac{c}{c+\frac{1}{2}\sigma^2\theta})^{n+1})=\mathcal{L}^{-1}((\frac{2c/\sigma^2}{2c/\sigma^2+\theta})^{n+1})=\frac{(\frac{2c}{\sigma^2})^{n+1}x^n}{n!}e^{-\frac{2c}{\sigma^2}x}$ is the probability density function of Erlang distribution with parameters of $(n+1,2c/\sigma^2)$ and
\begin{equation*}
\begin{split}
 \mathcal{L}^{-1}(e^{\frac{\lambda\alpha_i\delta_i}{\delta_i+\theta}}-1)=\frac{e^{-\delta_i x}\sqrt{\lambda\alpha_i\delta_i}\cdot I_1(2\sqrt{\lambda\alpha_i\delta_i x})}{\sqrt{x}},
\end{split}
\end{equation*}
If we let $F_{n+1}(x)=\frac{(\frac{2c}{\sigma^2})^{n+1}x^n}{n!}e^{-\frac{2c}{\sigma^2}x}$, $G_i(x)=\frac{e^{-\delta x}\sqrt{\lambda\alpha_i\delta_i}\cdot I_1(2\sqrt{\lambda\alpha_i\delta_i x})}{\sqrt{x}}$, $\pi(x)=1$ on $[0, +\infty)$ and zero otherwise. 
\end{proof}

\begin{remark}
When $N=2$, 
\begin{equation*}
\begin{split}
 W^{(q)}(x)=\pi*\sum_{n=0}^{\infty}(\frac{\lambda(1-e^{-\lambda})+q}{c}\pi-\frac{\lambda e^{-\lambda}}{c}\pi*(G_1+G_2+G_1*G_2))^{*n}*F_{n+1}(x).
\end{split}
\end{equation*}
\end{remark}


\section{Appendix}
\subsection{The moments of an MIPP}

From (\ref{71}), we can easily calculate the moments of $V_t^{(n)}$,  as follows:
\begin{equation}\label{73}
E[V_t^{(n)}]^m=\sum_{k=0}^{+\infty}P(V_t^{(n-1)}=k)B_m(\lambda k).
\end{equation}
The Bell polynomial has the property that $B_m(x)=x\sum_{j=0}^{m-1}\binom{m-1}{j}B_j(x)$, where  $B_0(x)=1$ and $B_1(x)=x$. 

i) When $\lambda\neq1$, we obtain the first four moments by induction.
\begin{equation}\label{equ:square201}
\begin{aligned}
 E[V_{t}^{(n)}]&=\lambda^{n}t,\\
 E[V_{t}^{(n)}]^{2}&=\lambda^{2n}t^{2}+\frac{1-\lambda^{n}}{1-\lambda}\lambda^{n}t,\\
 E[V_{t}^{(n)}]^{3}&=\lambda^{3n}t^{3}+\frac{3\lambda^{2n}t^{2}(1-\lambda^{n})}{1-\lambda}+\frac{3\lambda^{n+1}t(1-\lambda^{2n-2})}{(1-\lambda)(1-\lambda^{2})}-\frac{3\lambda^{2n}t(1-\lambda^{n-1})}{(1-\lambda)^{2}}+\frac{\lambda^{n}t(1-\lambda^{2n})}{1-\lambda^{2}},\\
 E[V_{t}^{(n)}]^{4}&=\lambda^{4n}t^{4}+\frac{6\lambda^{3n}t^{3}(1-\lambda^{n})}{1-\lambda}+\frac{18\lambda^{2n+1}t^{2}(1-\lambda^{2n-2})}{(1-\lambda)(1-\lambda^{2})}-\frac{18\lambda^{3n}t^{2}(1-\lambda^{n-1})}{(1-\lambda)^{2}}\\
 &+\frac{18\lambda^{n+3}t(1-\lambda^{3n-3})}{(1-\lambda)(1-\lambda^{2})(1-\lambda^{3})}-\frac{18\lambda^{3n-1}t(1-\lambda^{n-1})}{(1-\lambda)^{2}(1-\lambda^{2})}-\frac{18\lambda^{2n+1}t(1-\lambda^{2n-2})}{(1-\lambda)^{2}(1-\lambda^{2})}\\
 &+\frac{18\lambda^{3n-1}t(1-\lambda^{n-1})}{(1-\lambda)^{3}}+\frac{6\lambda^{n+2}t(1-\lambda^{3n-3})}{(1-\lambda^{2})(1-\lambda^{3})}-\frac{6\lambda^{3n}t(1-\lambda^{n-1})}{(1-\lambda)(1-\lambda^{2})}\\
 &+\frac{7\lambda^{2n}t^{2}(1-\lambda^{2n})}{1-\lambda^{2}}+\frac{7\lambda^{n+1}t(1-\lambda^{3n-3})}{(1-\lambda)(1-\lambda^{3})}-\frac{7\lambda^{2n}t(1-\lambda^{2n-2})}{(1-\lambda)(1-\lambda^{2})}+\frac{\lambda^{n}t(1-\lambda^{3n})}{1-\lambda^{3}}.
\end{aligned}
\end{equation}

Then, the variance, skewness and kurtosis can be calculated as
\begin{equation}\label{equ:square203}
\begin{aligned}
 Var(V_{t}^{(n)})&=E[V_{t}^{(n)}]^{2}-(E[V_{t}^{(n)}])^{2}=\frac{1-\lambda^{n}}{1-\lambda}\lambda^{n}t,\\
 Skew(V_{t}^{(n)})&=E\left[\left(\frac{V_{t}^{(n)}-E[V_{t}^{(n)}]}{\sigma}\right)^{3}\right]=\frac{\lambda^{n+1}+2\lambda^{n}-2\lambda-1}{(\lambda^{2}-1)\sqrt{\frac{\lambda^{n}-1}{\lambda-1}\lambda^{n}t}},\\
 Kurt(V_{t}^{(n)})&=E\left[\left(\frac{V_{t}^{(n)}-E[V_{t}^{(n)}]}{\sigma}\right)^{4}\right]=\frac{\left(\begin{aligned}
          &1+6\lambda+5\lambda^{2}+6\lambda^{3}-6\lambda^{n}+6\lambda^{2n}-12\lambda^{n+1}\\
          &-13\lambda^{n+2}-5\lambda^{n+3}+6\lambda^{2n+1}+5\lambda^{2n+2}+\lambda^{2n+3}
         \end{aligned}\right)}{\lambda^{n}(\lambda^{2}-1)(\lambda^{2}+\lambda+1)(\lambda^{n}-1)t}+3.
\end{aligned}
\end{equation}

ii) When $\lambda=1$, we can take the limit as $\lambda\to 1$ in (\ref{equ:square201}) and get

\begin{equation*}
\begin{aligned}
 E[V_{t}^{(n)}]&=t,\\
 Var(V_{t}^{(n)})&=nt,\\
 Skew(V_{t}^{(n)})&=\frac{3n-1}{2\sqrt{nt}},\\
 Kurt(V_{t}^{(n)})&=\frac{6n^{2}-5n+1}{2nt}+3.
\end{aligned}
\end{equation*}

iii) We can easily check that when $0<\lambda<1$ we have 
\begin{equation*}
\begin{split}
\lim_{n\to\infty}Skew(V_{t}^{(n)})=\lim_{n\to\infty}\frac{\lambda^{n+1}+2\lambda^n-2\lambda-1}{(\lambda^2-1)\sqrt{\frac{\lambda^n(\lambda^n-1)t}{\lambda-1}}}=\infty, \; \; \lim_{n\to\infty}Kurt(V_{t}^{(n)})=+\infty,
\end{split}
\end{equation*}
and when $\lambda>1$,  we have
\begin{equation*}
\begin{split}
\lim_{n\to\infty}Skew(V_{t}^{(n)})=\frac{\lambda+2}{(\lambda+1)\sqrt{(\lambda-1)t}}, \;\; \lim_{n\to\infty}Kurt(V_{t}^{(n)})=\frac{6+6\lambda+5\lambda^2+\lambda^3}{(\lambda^2-1)(\lambda^2+\lambda+1)t}.
\end{split}
\end{equation*}




\newpage
\printbibliography

\end{document}